\documentclass{amsart}

\usepackage{amssymb,amsbsy}
\usepackage{amsrefs}

\newtheorem{theorem}{Theorem}[section]
\newtheorem{lemma}[theorem]{Lemma}
\newtheorem{proposition}[theorem]{Proposition}
\newtheorem{corollary}[theorem]{Corollary}

\theoremstyle{definition}
\newtheorem{definition}[theorem]{Definition}

\theoremstyle{remark}

\numberwithin{equation}{section}
\newtheorem*{theorem*}{Theorem}

\newcommand{\1}{_{(1)}}
\newcommand{\2}{_{(2)}}
\newcommand{\3}{_{(3)}}
\newcommand{\4}{_{(4)}}

\newcommand{\basisB}{\mathcal{B}}
\newcommand{\fc}{\mathfrak c }
\newcommand{\fg}{\mathfrak g}
\newcommand{\fh}{\mathfrak h}
\newcommand{\fm}{\mathfrak m}
\newcommand{\fn}{\mathfrak n}
\newcommand{\fp}{\mathfrak p}
\newcommand{\fs}{\mathfrak s}
\newcommand{\field}{\boldsymbol{k}}
\newcommand{\R}{\mathbb R}
\newcommand{\T}{T}
\newcommand{\sinn}{\ensuremath{_{\rm inn}}}
\newcommand{\sred}{\ensuremath{_{\rm red}}}
\newcommand{\supopp}{\ensuremath{^{\rm opp}}}

\DeclareMathOperator{\ad}{ad}
\DeclareMathOperator{\Ad}{Ad}
\DeclareMathOperator{\core}{core}
\DeclareMathOperator{\Endo}{End}
\DeclareMathOperator{\gl}{\mathfrak{gl}}
\DeclareMathOperator{\Id}{\mathrm{Id}}
\DeclareMathOperator{\Ima}{lm}
\DeclareMathOperator{\Lie}{Lie}
\DeclareMathOperator{\Prim}{Prim}
\DeclareMathOperator{\spann}{span}
\font\cyrillic=wncyi10
\newcommand{\SU}{\mathop{\hbox{{\cyrillic UX}}}}

\author{J. M. P\'erez-Izquierdo}
\address{Departamento de Matem\'aticas y Computaci\'on, Universidad de
La Rioja, 26004 \\ Lo\-gro\-\~no, Spain}
\email{jm.perez@unirioja.es}
\thanks{J.M. P\'erez-Izquierdo  thanks support from the Spanish Ministerio de Ciencia e Innovaci\'on (MTM2013-45588-C3-3-P) and the Programa hispano--brasile\~no de cooperaci\'on interuniversitario (PHBP14/00110).}%

\begin{document}
\title{Hyporeductive and pseudoreductive Hopf algebras}
\begin{abstract}
 In his generalization of reductive homogeneous spaces, Lev Sabinin showed that Lie's fundamental theorems hold for local analytic hyporeductive and pseudoreductive  loops.  We derive Sabinin's results in an algebraic context in terms of  non-associative Hopf algebras that satisfy the analog of the hyporeductive and pseudoreductive identities for loops.
\end{abstract}
\maketitle
%
%
\section{Introduction}
The knowledge of homogeneous spaces firmly relays on their infinitesimal analysis. Symmetric  and reductive homogeneous spaces are archetypal examples. In 1954 Nomizu \cite{No54} studied invariant affine connections on reductive homogeneous spaces $G/H$ under this premise. In this context, \emph{reductive} means that the Lie tangent algebra $\fg$ of the connected Lie group $G$ splits as $\fg = \fh \oplus \fm$, where $\fh$ is the Lie algebra of the closed subgroup $H$, and $\fm$ is a complement satisfying $\Ad_H(\fm) \subseteq \fm$. Nomizu proved that invariant affine connections on these manifolds are in bijection with certain non-associative products on $\fm$ admitting $\Ad_H$ as automorphisms. Later, in \cite{Ya58}, Yamaguti introduced the notion of general Lie triple system as an abstract model for the complement $\fm$. Other names for these triple systems are Lie triple algebras  \cite{Ki75} and Lie-Yamaguti algebras \cite{Kin01}. The importance of the local study of homogeneous spaces was beyond doubt, and it was confirmed in the landmark paper \cite{Wo68} where Wolf described the geometry and structure of isotropy irreducible homogeneous spaces through heavy and extensive computations based on the representation theory of Lie algebras. A related, but different, approach to this classification was obtained recently in the spirit of coordinatizing Lie algebras by non-associative structures \cites{BeElMa09,BeElMa11}. 

A new approach to the local study of affine manifolds was proposed in \cite{Ki64} by Kikkawa: around any point $e$ of such a manifold $M$, there exists a locally defined product
\begin{equation}\label{eq:geodesic_sum}
xy:= \exp_x(\tau^e_x(\exp^{-1}_e(y)))
\end{equation} 
where $\tau^e_x$ denotes the parallel transportation from $e$ to $x$. Thus, it is possible to classify affine connections according to the algebraic properties of these products. For instance, around $\bar{e} := eH \in G/H$ ($e$ denotes the identity element of $G$) the reductive property of $G/H$ implies
\begin{equation}\label{eq:automorphisms}
h(\bar{x}\bar{y}) = (h\bar{x}) (h\bar{y})
\end{equation}
for all $h \in H$ and $\bar{x}:= xH, \bar{y}:= yH \in G/H$, where $G/H$ is equipped with its canonical connection \cite{No96}. This new approach was very much related to the  beautiful interplay among algebra and geometry in the study of symmetric spaces by Loos \cite{Lo69}--who, by the way, also discussed the idea of using non-associative Hopf algebras to understand certain products on manifolds--and it was sharpened by Sabinin in \cite{Sa77}. Sabinin devised a program to model affinely connected spaces in terms of algebraic structures, which also allowed him to consider discrete spaces. 

In the work of Sabinin, \emph{loops} (non-associative groups to say) play an important role. For instance, the geodesic sum defined in (\ref{eq:geodesic_sum}) is a (partial) loop. Thus, it was apparent that a Lie theory for loops was required. This theory finally appeared in \cite{SaMi87}, where tangent algebras of analytic loops (now called Sabinin algebras) were axiomatized, and Lie's fundamental theorems were proved in the non-associative case. After the strong relations of loops and affine connections, now it is not a surprise that Sabinin algebras were defined modeled on the torsion of flat affine connections.

The approach to affine connections in terms of loops has certain benefits. For instance, from this point of view (\ref{eq:automorphisms}) can be naturally generalized since not only the groups of automorphisms of loops are of interest. An \emph{autotopism} is a triple $(f_0,f_1,f_2)$ of invertible maps satisfying $f_0(xy) = f_1(x)f_2(y)$ for any elements $x$ and $y$ of the loop. Automorphisms correspond to autotopisms of the form $(f,f,f)$. Thus, instead of imposing that the elements of $H$ act as automorphisms of the geodesic sum, it is  natural to impose that they appear as components of  autotopisms. This motivated the introduction, among others, of \emph{hyporeductive} and \emph{pseudoreductive} loops by Sabinin as a generalization of geodesic sums on homogeneous reductive spaces. The Lie theory of hyporeductive loops appeared in \cites{Sa90a,Sa90b,Sa91}, while the correspoing theory for pseudoreductive loops waited until \cite{Sa96} (see \cite{Is95} for the description of the associated affine connections).

After decades of work on analytic loops with tools from differential geometry, tangent algebras of analytic loops were finally put into algebraic ground by Shestakov and Umirbaev in their landmark study of primitive elements of non-associative bialgebras  \cite{SU02}. In \cite{MoPe10} the geometric approach by Sabinin and the algebraic approach by Shestakov and Umirbaev were conciliated by means of non-associative Hopf algebras, as predicted by Loos forty years earlier. 

The use of non-associative Hopf algebras in the treatment of loops is, in our opinion, indeed valuable. For instance, a non-associative version of the Baker-Campbell-Haussdorf formula has been recently found \cite{MoPeSh16a}. The freeness of the loop of formal power series, with coefficient $1$, in non-associative variables has been proved with the help of non-associative Hopf algebras \cite{MoPeSh16b}. A Lie theory for commutative automorphic loops has also been obtained with these techniques \cite{GrPe17}, while no similar approach from differential geometry was available. Thus, the introduction of new algebraic techniques in this field makes it more accessible to other researches.

In this paper we will transport Sabinin's results on the Lie theory of hyporeductive and pseudoreductive loops to an algebraic setting. To that end, in sections \ref{sec:hHa} and \ref{sec:pHa} we will define hyporeductive and pseudoreductive Hopf algebras, we will describe the space of primitive elements in terms of some triples $(\fg, \fs, \fc)$ composed of a Lie algebra $\fg$, a Lie subalgebra $\fs$ and a complement $\fc$ satisfying certain properties, and finally we will show how to formally integrate these triples to recover hyporeductive and pseudoreductive Hopf algebras. Preliminaries on non-associative Hopf algebras and Sabinin algebras are included in Section~\ref{sec:prel}, while the general result on the formal integration of triples is proved in Section~\ref{sec:maHa}. We hope this paper will contribute to motivate classifications similar to those in \cite{Wo68}.

\subsection*{Notation} Given a Lie algebra $\fg$ and a subset $\fc \subseteq \fg$, the Lie subalgebra of $\fg$ generated by $\fc$ will be denoted by $\Lie_{\fg}\langle \fc \rangle$.  The subalgebra $N_{\fg}(\fc) := \{ x \in \fg \mid [x,\fc] \subseteq \fc\}$ will appear frequently. For a subalgebra $\fs$ of $\fg$, $\core_{\fg}(\fs)$ is, by definition, the largest ideal of $\fg$ contained in $\fs$. If there is no likelihood of confusion, we will write $N(\fs)$ instead of $N_{\fg}(\fs)$. The symmetric group of degree $n$ will be denoted by $S_n$. Since almost all algebras in this paper lack associativity, an order for parentheses in powers is mandatory. We will stick to the convention $c^n:= ((cc)\cdots )c$. Finally, all bialgebras in this paper are assumed to be coassociative, cocommutative and unital, and the characteristic of the base field $\field$ is zero.

%
\section{Preliminaries}
\label{sec:prel}

In this section we will review some basic results and definitions required in the present work. See \cite{MoPeSh14} for an expository paper on non-associative Hopf algebras.

\subsection{Loops.} A \emph{loop} $(Q,xy,e)$ is a non-empty set $Q$ endowed with:
\begin{itemize}
	\item a product $xy$ for which the left and right multiplication operators $L_x \colon y \mapsto xy$ and $R_x \colon y \mapsto yx$ are bijective for all $x \in Q$, and
	\item a distinguished element $e \in Q$, called the \emph{unit element} (or \emph{identity element}), such that $ex = x = xe$ for all $x \in Q$.
\end{itemize}
Since associativity is not required in the definition, this algebraic structure is usually thought as the non-associative counterpart of groups. The inverse map for groups is replaced in loop theory by the \emph{left} and \emph{right divisions}
\begin{displaymath}
	x \backslash y := L^{-1}_x(y) \quad \text{and} \quad x/y := R^{-1}_y(x).
\end{displaymath} 
Clearly
\begin{displaymath}
x \backslash (xy) = y = x(x \backslash y) \quad \text{and} \quad (xy)/y = x = (x/y)y.
\end{displaymath}

\subsection{Formal loops.} Differential manifolds with affine connections have traditionally been a natural source of loops. Given such a manifold $M$ and a point $e \in M$, then the geodesic sum
\begin{displaymath}
xy := \exp_x(\tau^{e}_{x} (\exp^{-1}_e(y))),
\end{displaymath}
where $\tau^e_x$ stands for the parallel transport from $e$ to $x$, defines a binary operation $U \times U \to M$ on a neighborhood $U$ of $e$. The point $e$ is the unit element of this operation, and the left and right multiplication operators by $x$ are invertible when restricted to 
 small neighborhoods of $e$. In case  $(x,y) \mapsto xy$ is analytic and we only study this product around $e$, we can assume $M = \R^n$, $e = (0,\dots, 0)$ and $xy = (p_1(x,y),\dots, p_n(x,y))$ where $p_i(x,y)$  are power series ($i=1,\dots, n$). The algebra of formal power series on $x_1,\dots, x_n$ can be identified with the dual space of the \emph{symmetric algebra} $\R[V]$ on $V:=\R^n$. Under this identifications, $xy$ is nothing else but a linear map $p \colon \R[V] \otimes \R[V] \to V$ for which $p(u \otimes 1)$ and $p(1 \otimes u)$ are both equal to the projection of $u$ onto $V$.

In an algebraic setting the convergence of the series $p_1(x,y),\dots,p_n(x,y)$ is relaxed, and analytic loops are replaced by formal loops. Given a vector space $V$ over a field $\field$, a \emph{formal loop} (over $V$) is a linear map $F\colon \field[V] \otimes \field[V] \to V$ such that \begin{displaymath}
F\vert_{\field[V] \otimes 1} = \pi_V = F\vert_{1 \otimes \field[V]}
\end{displaymath}
where $\field[V]$ stands for the symmetric $\field$-algebra on $V$, and $\pi_V$ denotes the projection of $\field[V]$ onto $V$ that kills homogeneous elements of degree $\neq 1$.

\subsection{Nonassociative connected bialgebras.} 
The category of formal loops over a field $\field$ of characteristic zero is equivalent to the category of connected non-associative bialgebras \cite{MoPe10}. Any formal loop $F$ defines a coalgebra morphism $F' := \exp^*(F)$ by
\begin{eqnarray*}
F' \colon \field[V] \otimes \field[V] & \rightarrow & \field[V] \\
u \otimes v & \mapsto & uv := \sum_{n=0}^{\infty} \frac{1}{n!} F(u\1 \otimes v\1) \circ \cdots \circ F(u_{(n)} \otimes v_{(n)})
\end{eqnarray*}
where the product $\circ$ in the formula is the usual commutative and associative product on $\field[V]$, that we will no use anymore. The \emph{comultiplication} $\Delta$ is the homomorphism of unital commutative and associative algebras $\Delta \colon \field[V] \rightarrow \field[V] \otimes \field[V]$ determined by $\Delta(a) = a \otimes 1 + 1 \otimes a$ for all $a \in V$. Sweedler's sigma notation for $\Delta(u)$ is $\Delta(u) = \sum u\1 \otimes u\2$ or, with the Einstein convention,  just $\Delta(u) = u\1 \otimes u\2$. The \emph{counit} $\epsilon \colon \field[V] \rightarrow \field$ is analogously determined by $\epsilon(a) = 0$ for all $a \in V$. Endowed with the new product $uv := F'(u \otimes v)$ and the unit $\eta \colon \field \rightarrow \field[V]$ ($\xi \mapsto \xi 1$) we obtain a non-associative connected  bialgebra $(\field[V], \Delta, \epsilon, F', \eta)$. The formal loop is recovered as $F = \pi_V F'$.

\subsection{Left and right divisions. Non-associative Hopf algebras} Associative connected bialgebras are Hopf algebras, i.e. for any such bialgebra there exists a map, the \emph{antipode}, related to the bialgebra structure by $S(u\1)u\2 = \epsilon(u)1 = u\1 S(u\2)$ for all $u$. This useful map is no longer available in the non-associative setting. However, the existence of two maps (the \emph{left} and the \emph{right division}) $\backslash, / \colon \field[V] \otimes \field[V] \rightarrow \field[V]$ satisfying
\begin{equation}\label{eq:divisions}
\begin{split}
u\1 \backslash (u\2 v) &= \epsilon(u) v = u\1 (u\2 \backslash v) \quad \text{and}\\
 \quad (u v\1)/v\2 &= \epsilon(v)u = (u / v\1) v\2
 \end{split}
\end{equation}
can be proved  for any non-associative connected bialgebra \cite{Pe07}. These are the non-associative counterpart of the antipode $S$. In the associative case one has $u \backslash v = S(u)v$ and $u/v = u S(v)$. In this paper we will only consider \emph{coassociative}  (i.e $u_{(1)(1)} \otimes u_{(1)(2)} \otimes u_{(2)} = u_{(1)} \otimes u_{(2)(1)} \otimes u_{(2)(2)}$)  and \emph{cocommutative} (i.e. $u\1 \otimes u\2 = u\2 \otimes u\1$) coalgebras. In this setting, a \emph{non-associative Hopf algebra} is a unital bialgebra with a left and a right division satisfying (\ref{eq:divisions}).  The adjective \emph{connected} means that the Hopf algebra is isomorphic to $(\field[V], \Delta, \epsilon, uv, \eta, \backslash, /)$ for some space $V$ (the space of primitive elements).

\subsection{Primitive elements and Sabinin algebras.} $V$ can be thought as the tangent space at the identity element of the formal loop $(\field[V],F)$. In terms of the comultiplication, $V$ is the nothing else but the vector space of \emph{primitive elements}, $\Prim \field[V]$, of $\field[V]$, i.e. those elements satisfying
\begin{displaymath}
\Delta(a) = a \otimes 1 + 1 \otimes a.
\end{displaymath}
The space of primitive elements is not closed under the product $uv$. However, many multilinear maps  for which $V$ is stable, such as the commutator $[u,v] := uv - vu$ or the associator $(u,v,w):=(uv)w-u(vw)$, can be obtained out of the binary product $uv$. The precise algebraic structure of $\Prim \field[V]$ was determined by Shestakov and Umirbaev \cite{SU02}. Later, in \cite{Pe07},  it was named \emph{Sabinin algebra} since it was first introduced by Sabinin and Mikheev as the right algebraic structure on the tangent space at the identity element to classify local loops. A Sabinin algebra is a rather complex algebraic structure endowed with two infinite families of multilinear operations satisfying certain axioms that generalize those of Lie algebras.

\subsection{Shestakov-Umirbaev operations}\label{subsec:SU_operations} Let  $X:= \{ x_1,x_2,\dots,  y, y_1,y_2,\dots, z\}$ be a set of symbols, and let $\field\{X\}$ be the free unital non-associative algebra generated by $X$. Define homomorphisms of unital algebras $\Delta \colon \field\{X\} \to \field\{X\} \otimes \field\{X\}$ and $\epsilon \colon \field\{ X \}\to \field$ determined by $\Delta(a) = a \otimes 1 + 1 \otimes a$ and $\epsilon(a) = 0$ for any $a \in X$. Consider $u:= ((x_1x_2)\cdots)x_m$, $v:= ((y_1y_2)\cdots)y_{n}$,  $\underline{u}:= x_1 \otimes \cdots \otimes x_m$ and $\underline{v}:= y_1 \otimes \cdots \otimes y_n$. The element $\Delta(u) = u\1 \otimes u\2$ has been defined, and the notation $\Delta(\underline{u}) = \underline{u}\1  \otimes \underline{u}\2$ has the obvious meaning. We can recursively define intermediate elements $p(\underline{u};\underline{v};z) = p(x_1,\dots, x_{m}; y_1,\dots, y_{n};z)$ from the associator $(u,v,w) := (uv)w-u(vw)$ by
\begin{equation*}
(u,v,z) = (u\1 v\1)p(\underline{u}\2;\underline{v}\2;z)
\end{equation*}
to obtain elements $\langle x_1, \dots, x_m; y,z\rangle$ and $\Phi(x_1,\dots, x_m; y_1,\dots, y_{n};y_{n+1})$ ($m, n \geq 1$) by 
\begin{equation*}
\begin{split}
& \langle 1; y,z \rangle := -[y,z], \\
& \langle x_1,\dots, x_m; y,z \rangle := -p(\underline{u};y;z) + p(\underline{u};z;y) \quad \text{and}\\
&\Phi(x_1,\dots, x_n; y_{1},\dots, y_{m};y_{m+1}) := \\
& \quad \frac{1}{n! (m+1)!} \sum_{\sigma \in S_n, \tau \in S_{m+1}}  p(x_{\sigma(1)}, x_{\sigma(2)},\dots, x_{\sigma(n)}; y_{\tau(1)}, y_{\tau(2)}, \dots, y_{\tau(m)};  y_{\tau(m+1)} )
\end{split}
\end{equation*}
where $S_k$ is the symmetric group or degree $k$. These non-associative polynomials  $\langle x_1, \dots, x_n; y,z\rangle$ and $\Phi(x_1,\dots, x_n; y_1,\dots, y_{m};y_{m+1})$ can be evaluated on any non-associative algebra $H$, thus they define multilinear maps on $H$ called \emph{Shestakov-Umirbaev operations}. With these new operations $H$ becomes a Sabinin algebra denoted by $\SU(H)$, and in case $H$ is a non-associative Hopf algebra then $\Prim(H)$ is Sabinin subalgebra of $\SU(H)$. In Section~\ref{subsec:Sabinin} we will use this construction.
 
\subsection{Formal integration.} Formal loops are equivalent to non-associative connected Hopf algebras. Thus, the formal integration of a Sabinin algebra to a formal loop amounts to constructing such a Hopf algebra out of a given Sabinin algebra. The only requirement for this construction is that the Sabinin algebra structure is recovered from the Hopf algebra as the space of primitive elements endowed with the Shestakov-Umirbaev operations. For instance, the formal integration of a Lie algebra corresponds essentially to the construction of its universal enveloping algebra and the proof of the Poincar\'e-Birkhoff-Witt Theorem. 

In the approach to hyporeductive and pseudoreductive local loops, triples consisting of a Lie algebra, a subalgebra and a complement of this subalgebra overshadow Sabinin algebras. In fact, all multilinear operations conforming the Sabinin algebra structure of the tangent space at the identity element of any of these loops can be obtained from just three products of arieties $2,2$ and $3$ derived from a certain triple related to the fundamental vector fields of an adequate affine connection; hence, no need for the whole power of Sabinin algebras is required. We  adopt this point of view in this paper, and we formally integrate triples rather than Sabinin algebras. 

%
\section{Right monoalternative Hopf algebras.}
\label{sec:maHa}

\subsection{Triples} Let us consider a category whose objects are \emph{triples} $\tau = (\fg, \fs, \fc)$ where $\fg$ is a Lie algebra, $\fs$ is a subalgebra of $\fg$, and $\fc$ is a subspace such that $\fg = \fs \oplus \fc$. A morphism $\varphi \colon \tau \to \tau'$ is a homomorphism of Lie algebras $\varphi \colon \fg \to \fg'$ satisfying $\varphi(\fs) \subseteq \fs'$ and $\varphi(\fc) \subseteq \fc'$. Associated to any triple $\tau$ we have two more triples $\tau\sinn$ and 
$\tau\sred$ defined by $\tau\sinn = (\fg\sinn, \fs\sinn, \fc)$, where $\fg\sinn := \Lie_{\fg}\langle \fc \rangle$, $ \fs\sinn := \fg\sinn \cap \fs$ and $\tau\sred := (\fg\sinn/\core(\fs\sinn), \fs\sinn / \core(\fs\sinn), \fc)$\footnote{$\core(\fs\sinn):=\core_{\fg\sinn}(\fs\sinn)$.}. We will say that the triples $\tau$ and $\tau'$ are \emph{equivalent} if $\tau\sred$ is isomorphic to $\tau'\sred$.

\subsection{Triples from bialgebras} Any unital bialgebra $U$ determines a triple 
\begin{equation*}
\T(U):= (\fg(U), \fs(U), \fc(U))
\end{equation*}
where 
\begin{align*}
	 \fg(U):= & \{ f \in \Endo_{\field}(U) \mid \Delta(f(u)) = f(u\1) \otimes u\2 + u\1 \otimes f(u\2)\,\, \forall_{u \in U}\}, \\
	 \fs(U) := & \{ f \in \fg(U) \mid f(1) = 0 \} \quad\text{and} \\ 
	 \fc(U) := & \{ R_a \mid a \in \Prim(U)\}.
\end{align*}
\begin{lemma}
Let $U$ be a unital connected bialgebra. Then, $\T(U)\sinn = \T(U)\sred$.
\end{lemma}
\begin{proof}
Unital connected bialgebras are linearly spanned by elements $((c_1c_2)\cdots)c_n$ where $c_1,\dots, c_n \in \Prim(U)$ and $n \geq 0$ \cite{SU02}. We will show  by induction on $n$ that any element in $\core(\fs(U)\sinn)$ kills these generators. The case $n = 0$ is clear by the definition of $\fs(U)$. The general case follows from $f( uc_n) = fR_{c_n}(u) = [f,R_{c_n}] (u) + R_{c_n}f(u)$ for any $f \in \core(\fs(U)\sinn)$ and $u =((c_1c_2) \cdots )c_{n-1}$. Thus, $\core(\fs(U)\sinn) = \{ 0 \}$.
\end{proof}

\subsection{Formal integration of triples} A \emph{right monoalternative} bialgebra is a bialgebra satisfying
\begin{equation*}
	((xy\1)\cdots) y_{(n)} = x((y\1 y\2)\cdots y_{(n)})
\end{equation*}
for any $x, y$ and any $n \geq 0$. The following result can be essentially found in \cites{MoPe10,Pe07}, although there it is formulated in terms of Sabinin algebras. An approach through these algebras is natural when the families of Hopf algebras under consideration are rooted to varieties of loops.
\begin{theorem}
Let $\tau = (\fg, \fs, \fc)$ be a triple. We have:
\begin{enumerate}
\item There exist a right monoalternative (connected) Hopf algebra $U_{\tau}$ and a morphism $\iota \colon \tau \to \T(U_{\tau})$ whose restriction to $\fc$ gives a linear isomorphism of $\fc$ and $\fc(U_{\tau})$.
\item For any monoalternative unital bialgebra $U'$ and any morphism $\iota' \colon \tau \to \T(U')$ there exists a unique homomorphism of bialgebras $\psi \colon U_{\tau} \to U'$ such that $\iota'_{c} (1) = \psi(\iota_c(1))$ for any $c \in \fc$.
\item $\tau$ and $\T(U_{\tau})$ are equivalent.
\item Given triples $\tau$ and $\tau'$, $U_{\tau}$ is isomorphic to $U_{\tau'}$ if and only if $\tau$ is equivalent to $\tau'$.
\end{enumerate}
\end{theorem}
\begin{proof}
Consider the universal enveloping algebra $U(\fg)$ of $\fg$. The left ideal $U(\fg)\fs$ and the left module $U_{\tau} := U(\fg) / U(\fg) \fs$, which is also a coalgebra isomorphic to $\field[\fc]$. The element $u + U(\fg)\fs \in U_{\tau}$ will be denoted by $\bar{u}$. Let us fix a basis $\basisB$ of $\fc$. By the Poincar\'e-Birkhoff-Witt theorem, a basis of $U_{\tau}$ consists of elements
\begin{equation*}
\sum_{\sigma \in S_n}  \overline{c_{\sigma(1)} \cdots c_{\sigma(n)}} 
\end{equation*}
where $c_1,\dots, c_n \in \basisB$ and $n \geq 0$. We use this basis to define a product on $U_{\tau}$ as follows:
\begin{equation}\label{eq:productU}
\bar{u} \left( \sum_{\sigma \in S_n}  \overline{c_{\sigma(1)} \cdots c_{\sigma(n)}}  \right) := \sum_{\sigma \in S_n} \overline{c_{\sigma(n)} \cdots c_{\sigma(1)} u}.
\end{equation}
Notice that the indexes in the right-hand side of the previous equality can be arranged in any order. Clearly
\begin{equation}\label{eq:symmetrical_distribution}
\bar{u} \left( \sum_{\sigma \in S_n}  \overline{c_{\sigma(1)} \cdots c_{\sigma(n)}}  \right)  = \sum_{\sigma \in S_n} ((\bar{u}\overline{c_{\sigma(1)}})\cdots )\overline{c_{\sigma(n)}}.
\end{equation}
It is easy to check that endowed with this product and the coalgebra structure previously mentioned, $U_{\tau}$ is a connected unital bialgebra. Property (\ref{eq:symmetrical_distribution}) is in fact independent of the chosen basis $\basisB$, i.e. for any other basis $\basisB'$ of $\fc$ we also have
\begin{equation*}
\bar{u} \left( \sum_{\sigma \in S_n}  \overline{c'_{\sigma(1)} \cdots c'_{\sigma(n)}}  \right)  = \sum_{\sigma \in S_n} ((\bar{u}\overline{c'_{\sigma(1)}})\cdots )\overline{c'_{\sigma(n)}}
\end{equation*}
for any $c'_1,\dots, c'_n \in \basisB'$. Thus
\begin{equation}\label{eq:power_distribution}
	\bar{u} \bar{c}^n = ((\bar{u} \bar{c}) \cdots ) \bar{c}
\end{equation} 
for any $c \in \fc$. Since $\{\bar{c}^n \mid c \in \fc, \, n \geq 0\}$ linearly spans $U_{\tau}$, and $\bar{c}^n\1 \otimes \cdots \otimes \bar{c}^n_{(m)} = \sum_{i_1+\cdots + i_m = n} \frac{n!}{i_1!\cdots i_m!} \bar{c}^{i_1} \otimes \cdots \otimes \bar{c}^{i_m}$ then we can conclude that $U_{\tau}$ is right monoalternative. 

The action of $\fg$ on $U_{\tau}$ defines a homomorphism of Lie algebras $\iota \colon \fg \to \gl(U_{\tau})$ with
\begin{align*}
\Delta(\iota_x(\bar{u})) &= \Delta (\overline{xu}) = \overline{x u\1} \otimes \overline{u\2} + \overline{u\1} \otimes \overline{x u \2} \\
&= \iota_x(\overline{u\1}) \otimes \overline{u\2} + \overline{u\1} \otimes \iota_x(\overline{u\2}).
\end{align*}
Hence, $\iota$ determines a homomorphism of Lie algebras $\iota \colon \fg \to \fg(U_{\tau})$.  Moreover, for any $d \in \fs$ and $c \in \fc$, $\iota(d)(\bar{1}) = \bar{d} = \bar{0}$ and $\iota_c(\bar{u}) = \overline{cu} = \bar{u} \bar{c} = R_{\bar{c}}(\bar{u})$. Hence, $\iota$ is a morphism from $\tau$ to $\T(U_{\tau})$.  The restriction of $\iota$ to $\fc$ is a linear isomorphism of $\fc$ and $\{R_{\bar{c}} \mid c \in \fc\} = \{ R_{\bar{c}} \mid \bar{c} \in \Prim(U_{\tau})\} = \fc(U_{\tau})$. 

Now we will prove (2). Let $U'$ be a (right) monoalternative unital bialgebra and $\iota' \colon \tau \to \T(U')$ a morphism. We can extend $\iota'$ to a homomorphism of unital associative algebras $\varphi \colon U(\fg) \to \Endo_{\field}(U')$ to obtain a linear map $\psi \colon U(\fg) \to U'$ defined by $\psi(u) : = \varphi_u(1)$, where $\varphi_u$ denotes the image of $u$ under $\varphi$. For any $d \in \fs$ and any $u \in U(\fg)$ we have $\psi(ud) = \varphi_{ud}(1) = \varphi_u\varphi_d(1) =  \varphi_u(\iota'_d(1)) =0$. Therefore, $\psi$ induces a linear map $\psi \colon U_{\tau} \to U'$ which satisfies
\begin{align*}
\psi(\bar{u}\bar{c}) &= \psi(\overline{cu}) = \psi(cu) = \varphi_{cu}(1) = \varphi_c\varphi_u(1) = R_{\varphi_c(1)}\varphi_u(1) = \varphi_{u}(1) \varphi_c(1)\\
&= \psi(\bar{u}) \psi(\bar{c}),
\end{align*}
for any $c \in \fc$ and $u \in U(\fg)$. Together with the right monoalternativity of $U_{\tau}$ and $U'$--recall (\ref{eq:power_distribution})--this proves that $\psi$ is a homomorphism of algebras. With little extra effort we get that $\psi$ is also a homomorphism of unital bialgebras. 

For any $c \in \fc$, $\iota'_c = R_{\iota'_c(1)} = R_{\varphi_c(1)} = R_{\psi(c)} = R_{\psi(\bar{c})} = R_{\psi(\iota_c(\bar{1}))}$, and since $\{ \iota_c(\bar{1}) \mid c \in \fc\}$ generates the algebra $U_{\tau}$, $\psi$ is the unique homomorphism $U_{\tau} \to U'$ that satisfies this property. 

In order to prove (3) we consider the morphism $\iota \colon \tau \to \T(U_{\tau})$. It restricts to a surjective homomorphism  $\fg\sinn \to \fg(U_{\tau})\sinn$ whose kernel is contained in $\core(\fs\sinn)$. In fact, $\core(\fs\sinn)$ is also contained in this kernel since the image of $\core(\fs\sinn)$ is an ideal of $\fg(U_{\tau})\sinn$ living in $\fs(U_{\tau})\sinn$, and the only such an ideal is $\{0\}$. This proves the isomorphism $\tau\sred \cong \T(U_{\tau})\sinn = \T(U_{\tau})\sred$. 

The statement in item (4) will follow from the existence of isomorphisms $U_{\tau} \cong U_{\T(U_{\tau})\sinn} \cong U_{\tau\sred}$ that we proceed to prove. As in item (2), $U_{\tau}$ is a $\fg(U_{\tau})\sinn$-module and we have a map $\psi \colon U(\fg(U_{\tau})\sinn) \to U_{\tau}$ with $\psi(f) : = f(\bar{1})$ which, thanks to the monoalternativity, induces a homomorphism of unital bialgebras $\psi \colon U_{\T(U_{\tau})\sinn} \to U_{\tau}$. On the one hand, $ U_{\T(U_{\tau})\sinn} $ has a basis $\{ \sum_{\sigma \in S_n} \overline{R_{c_{\sigma(n)}}\cdots R_{c_{\sigma(1)}}} \mid c_1,\dots, c_n \in \fc, \, n \geq 0\}$; on the other hand, $\{ \sum_{\sigma \in S_n} \overline{c_{\sigma(1)} \cdots c_{\sigma(n)}} \mid c_1,\dots, c_n \in \fc, \, n\geq 0 \}$ is a basis of $U_{\tau}$, and $\psi$ bijectively maps the former  basis onto the latter; thus $\psi$ is an isomorphism.  This shows the existence of an isomorphism  $U_{\tau} \cong U_{\T(U_{\tau})\sred}$. The isomorphism $U_{\T(U_{\tau})\sred} \cong U_{\tau\sred}$ is induced by the isomorphism $\T(U_{\tau})\sred \cong \tau\sred$.

Let us assume that $\tau$ and $\tau'$ are equivalent triples, i.e. $\tau\sred \cong \tau'\sred$. Thus, $U_{\tau\sred} \cong U_{\tau'\sred}$ and, by the previous paragraph, $U_{\tau}$ is isomorphic to $U_{\tau'}$. Conversely, if  $U_{\tau}$ is isomorphic to $U_{\tau'}$ then $\T(U_{\tau})\sred$ is isomorphic to $\T(U_{\tau'})\sred$, which leads to an isomorphism $\tau\sred \cong \tau'\sred$. 
\end{proof}
While in the previous proof we have used the notation $\bar{u}$ for elements in $U_{\tau}$, once this algebra is obtained there is no need for writing the bar, and we will do not so. Even more, it will become apparent that it is much more convenient to consider an isomorphic copy $\rho_{\fc}$ of $\fc$ as being a part of the triple $\tau$, while keeping $\fc$ for the space of primitive elements of $U_{\tau}$; thus, 
\begin{equation*}
	\tau =( \fg, \fs, \rho_{\fc}) \quad\text{and}\quad \fc = \Prim(U_{\tau}).
\end{equation*}
The conceptual benefit of this change is that now the morphism $\iota \colon \tau \to \T(U_{\tau})$ maps $\rho_c$ to $R_c$, i.e.  $\rho_c$ is sought as the abstract counterpart of $R_c$, while $\fg$ behaves as an abstract version of, and probably larger than, the Lie algebra generated by the right multiplication operators by primitive elements in $U_{\tau}$. The action of $\fg$ on $U_{\tau}$ extends to a homomorphism $U(\fg) \to \Endo_{\field}(U_{\tau})$. Again, although not totally correct, it would be helpful to use the same symbol to denote both, the element in $U(\fg)$ and its image in $\Endo_{\field}(U_{\tau})$. Thus, we will freely write $f(u)$ for elements $f \in U(\fg)$ and $u \in U_{\tau}$. 

\subsection{Natural factorization of $U(\fg)$} For any triple $\tau = (\fg, \fs, \rho_{\fc})$ we will show that a natural factorization of $U(\fg)$, needed in Section~\ref{sec:hHa}, exists. The following lemma follows easily from the existence of Poincar\'e-Birkhoff-Witt bases.
\begin{lemma}\label{lem:kernel}
Let $\tau = (\fg, \fs, \rho_{\fc})$ be a triple. Then, $U(\fs)$ is the largest subcoalgebra contained in $\{ f  \in U(\fg) \mid f(1) = \epsilon(f) 1 \}$.
\end{lemma}
Now we define elements $\rho_u$ ($u \in U_{\tau}$) in $U(\fg)$ that play the role of the right multiplication operators $R_u$ ($u \in U_{\tau}$). Let us fix a basis $\basisB$ of $\fc$, and recall the basis
\begin{equation*}
	\left\{ \sum_{\sigma \in S_n} ((c_{\sigma(1)} c_{\sigma(2)}) \cdots ) c_{\sigma(n)} \mid c_1,\dots, c_n \in \basisB, \, n \geq 0 \right\}
\end{equation*} 
of $U_{\tau}$. Define
\begin{equation}\label{eq:rho}
\rho_{\sum_{\sigma \in S_n} ((c_{\sigma(1)} c_{\sigma(2)}) \cdots ) c_{\sigma(n)}} := \sum_{\sigma \in S_n} \rho_{c_{\sigma(n)}} \cdots \rho_{c_{\sigma(1)}} \in U(\fg)
\end{equation}
for any $c_1,\dots, c_n \in \fc$ and any $n \geq 0$. We obtain in this way an injective morphism of coalgebras $U_{\tau} \to U(\fg)$ defined by $u \mapsto \rho_u$. Clearly
\begin{equation*}
	\rho_u(x) = xu \quad \forall u, x \in U_{\tau}.
\end{equation*}
\begin{proposition}
Let $\tau = (\fg, \fs, \rho_{\fc})$ be a triple. The coalgebra morphism
\begin{align*}
\pi \colon U(\fg) & \to  U(\fg) \\
f & \mapsto S(\rho_{f\1(1)})f\2 
\end{align*}
is a projection of $U(\fg)$ onto $U(\fs)$.
\end{proposition}
\begin{proof}
On the one hand, $\Ima \pi$ is a subcoalgebra of $U(\fg)$ and any $f \in \Ima \pi$ satisfies $\pi(f)(1) = S(\rho_{f\1(1)})f\2 (1) = S(\rho_{f\1(1)})\rho_{f\2 (1)}(1) = \epsilon(\rho_{f(1)}) 1 = \epsilon(f) 1$; thus, by Lemma~\ref{lem:kernel}, $\Ima \pi \subseteq U(\fs)$. On the other hand, for any $f \in U(\fs)$, $\pi(f) = S(\rho_{f\1(1)})f\2 = S(\rho_{\epsilon(f\1)1})f\2 = \epsilon(f\1)f\2 = f$. Therefore, $\pi$ is a projection of $U(\fg)$ onto $U(\fs)$.
\end{proof}
The following corollary is an immediate consequence of 1) the existence of Poincar\'e-Birkhoff-Witt bases for $U(\fg)$, and 2) the way we defined $\rho_u$ in (\ref{eq:rho}). 
\begin{corollary}
Let $\tau =(\fg, \fs, \rho_{\fc})$ be a triple. The mapping $f \mapsto \rho_{f\1(1)} \pi(f\2)$ defines a coalgebra isomorphism $U(\fg) \cong \rho_{U_{\tau}} \otimes_{\field} U(\fs)$.
\end{corollary}

%
\section{Hyporeductive Hopf algebras}
\label{sec:hHa}
In this section we will introduce hyporeductive Hopf algebras by analogy with hyporeductive loops. We will describe  hyporeductive triples associated to these Hopf algebras, closely following Sabinin's work, and finally we will formally integrate hyporeductive triples to obtain the corresponding hyporeductive Hopf algebras. 

\subsection{Hyporeductive loops, triples and Hopf algebras} A loop $Q$ is called (right) \emph{hyporeductive}\footnote{For analytic loops this definition agrees with the definition in \cites{Sa90a,Sa99}.} if the following properties hold:
\begin{itemize}
\item \emph{Right hypospeciality}: there exists a map $\circ \colon Q \times Q \rightarrow Q$ such that $u \circ e = e = e \circ u$, where $e$ stands for the identity element of $Q$, and \[ H(u,v)(xy) = H(u,v)(x) H'(u,v)(y),\] where by definition $H(u,v)(x) := r(u,v)(x) (u \circ v)$, 
\begin{equation*}
r(u,v)(x) := ((xu)v)/(uv) 
\end{equation*}
and $H'(u,v)(y) := (u \circ v) \backslash H(u,v)(y)$.
\item \emph{Right monoalternativity}: $((xy)y \cdots )y  = x((yy) \cdots y)$ for all $n \geq 0$, where $y$ appears $n$ times in each side of the equality.
\end{itemize}
\begin{definition}
A triple $\tau = (\fg, \fs, \rho_{\fc)}$ is called \emph{hyporeductive} if $\fg = N_{\fg}(\rho_{\fc}) + \rho_{\fc}$. 
\end{definition}
The Lie algebra generated by the fundamental vector fields of any analytic hyporeductive loop determines such a triple \cite{Sa99}. 
\begin{definition}
A non-associative Hopf algebra $U$ is called \emph{hyporeductive} if the following properties hold: 
\begin{itemize}
	\item \emph{Right hypospeciality}: there exists a coalgebra morphism $\circ \colon U \otimes U\rightarrow U$ such that $u \circ 1 = \epsilon(u)1 = 1 \circ u$ and
	\begin{displaymath}
			H(u,v) (xy) = H(u\1,v\1)(x) H'(u\2,v\2)(y),
	\end{displaymath}
	where by definition $H(u,v)(x) := r(u\1,v\2)(x) (u\2 \circ v\2)$, 
	\begin{equation*}
		r(u,v)(x) := ((x u\1) v\1)/(u\2 v\2)
	\end{equation*} 
	and $
	 H'(u,v)(y) := (u\1 \circ v\1) \backslash H(u\2,v\2)(y)$
	\item \emph{Right monoalternativity}: $(((xy\1)y\2)\cdots )y_{(n)} = x ((y\1 y\2)\cdots y_{(n)})$ for all $n \geq 0$.
\end{itemize}	
\end{definition}
Observe that for any $a \in \Prim(U)$ the monoalternativity implies $r(a,a) = R^2_a - R_{a^2} = 0$. Hence,
\begin{align}\label{eq:r}
	2r(a,b) &= r(a,b) - r(b,a) = R_bR_a - R_{ab} - R_aR_b + R_{ba} \\
	& = - [R_a,R_b] - R_{[a,b]} \nonumber
\end{align}
for all $a, b \in \Prim(U)$.
\begin{lemma}\label{lem:primitive}
	Let $U$ be a hyporeductive Hopf algebra. For any $a,b \in \Prim(U)$, the element $a \circ b$ is primitive.
\end{lemma}
\begin{proof}
	Since $\circ$ is a coalgebra morphism and $1 \circ a = 0 = a \circ 1$ for any primitive $a$, then $\Delta( a \circ b) = (a \circ b) \otimes 1 + (a \circ 1) \otimes (1 \circ b) + (1 \circ b) \otimes (a \circ 1) + 1 \otimes (a \circ b) = (a \circ b) \otimes 1 + 1 \otimes (a \circ b)$.
\end{proof}
\begin{proposition}
Let $U$ be a hyporeductive Hopf algebra. Then, $\T(U)\sinn$ is a hyporeductive triple.
\end{proposition}
\begin{proof}
On the one hand, given $a, b \in \Prim(U)$, the hypospeciality of $U$ implies $H(a,b)(xy) = H(a,b)(x) y  + xH'(a,b)(y)$, and hence $[H(a,b),R_{y}] = R_{H'(a,b)(y)}$; on the other hand, by (\ref{eq:r}), $-2H(a,b) = - 2r(a,b) - 2R_{a \circ b} = [R_a,R_b] + R_{[a,b] - 2 a \circ b}$. Therefore, 
\begin{equation*}
	[[R_a,R_b] + R_{[a,b] - 2 a \circ b}, R_c] = R_{-2H'(a,b)(c)}.
\end{equation*}
This proves 1) $\fg(U)\sinn = \spann\langle R_a , [R_a, R_b] \mid a,b \in \Prim(U) \rangle$, and 2) $\fg(U)\sinn = N_{\fg(U)\sinn}(\fc(U)) + \fc(U)$.
\end{proof}

\subsection{Formal integration of hyporeductive triples} We will prove that for any hyporeductive triple $\tau$, the non-associative Hopf algebra $U_\tau$ is hyporeductive. To this end, let us fix a hyporeductive triple $\tau = (\fg, \fs, \rho_{\fc})$ and let us denote $N_{\fg} (\rho_{\fc})$ just by $\fn$. 
\begin{lemma}\label{lem:section}
	There exists a coalgebra morphism $\sigma \colon U(\fs) \to U(\fn)$ such that $ \pi \sigma$ is the identity map of $U(\fs)$.
\end{lemma}
\begin{proof}
We include a proof for completeness. Since $\fg = \fs \oplus \rho_{\fc} = \fn + \rho_{\fc}$, then we can chose a linear map $\theta\colon U(\fs) \to \fn$ such that $\pi \theta(d) = d$ for any $d \in \fs$, and $\theta(1) = 0$. Let us extend this map to a coalgebra morphism $\theta'\colon U(\fs) \to U(\fn)$ by $\theta'=\exp^*(\theta)$, where
\begin{equation*}
	\exp^*(\theta)(x) := \sum_{n=0}^{\infty} \frac{1}{n!} \theta(x\1) \cdots \theta(x_{(n)}).
\end{equation*}
The coalgebra morphism $\Psi:= \pi \theta' \colon  U(\fs) \to U(\fs)$ is an isomorphism since its restriction to $\fs$ is the identity map of $\fs$ \cite{Sw69}*{Theorem 12.2.6}; in fact, $\Psi(d) = \pi\theta'(d) = \pi\theta(d) = d$ for all $d \in \fs$. A  coalgebra morphism that fulfills the requirement in the statement is $\sigma := \theta'\Psi^{-1}$. 
\end{proof}
For the coalgebra morphism $\sigma$ in Lemma~\ref{lem:section} we have
 \begin{equation}\label{eq:fundamental}
 \sigma (f) = \rho_{\sigma (f\1)(1)} \pi \sigma (f\2) = \rho_{\sigma (f\1)(1)} f\2.
 \end{equation}
\begin{theorem}
Let $\tau$ be a hyporeductive triple. Then, $U_{\tau}$ is a hyporeductive Hopf algebra.
\end{theorem}
\begin{proof}
We first define the map $\circ$ required for the hypospeciality. To this end, observe that the elements $\rho(u,v) := S(\rho_{u\1 v\1} ) \rho_{v\2}\rho_{u\2}$ ($u, v \in U_{\tau}$) span a coalgebra and they satisfy 
 \begin{equation*}
 \rho(u,v)(1) = S(\rho_{u\1 v\1})(u\2 v\2) =  S(\rho_{u\1 v\1})\rho_{u\2 v\2}(1) = \epsilon(u) \epsilon(v) 1.
 \end{equation*}
 Thus, by Lemma~\ref{lem:kernel}, $\rho(u,v) \in U(\fs)$. This ensures that the map
\begin{equation*}
u  \circ v := \sigma(\rho(u,v))(1)
\end{equation*}
is well-defined. Clearly, $u \circ 1 = \epsilon(u)1$, $1 \circ v = \epsilon(v) 1$ and
\begin{equation} \label{eq:fundamental_factorization}
\sigma( \rho(u,v)) = \rho_{u\1 \circ v\1} \rho(u\2, v\2) \in U(\fn).
\end{equation}
The hypospeciality of $U_{\tau}$ will be a consequence of 
\begin{equation*}
\rho_{u\1 \circ v\1} \rho(u\2, v\2) \rho_{U_{\tau}} S(\rho_{u\3 \circ v\3} \rho(u\4, v\4)) \subseteq \rho_{U_{\tau}}
\end{equation*}  
or, more generally, it will follow from $f\1 \rho_{U_{\tau}}  S(f\2) \subseteq \rho_{U_{\tau}} $ for all $f \in U(\fn)$. The advantage  is that we only have to prove it for  generators in $\fn$ of $U(\fn)$. Clearly $[f, \rho_{\fc}] \subseteq \rho_{\fc}$ by definition of $\fn$. In general,
\begin{equation*}
[f, \rho_{c^n}] = [f, \rho^n_{c}] = \sum \rho_c \cdots \rho_c \rho_{f'(c)}\rho_c \cdots \rho_c
\end{equation*}
for some $f'(c) \in \fc$. Hence, the linearization of $\rho_{c^n} = \rho^n_c$ leads to $[f,\rho_{c^n}] \in \rho_{U_{\tau}}$.

Finally, take $f := \rho_{u\1 \circ v\1} \rho(u\2, v\2)$. Evaluating at $1$ we obtain $f\1 \rho_{y} S(f\2) = \rho_{f\1 \rho_y S(f\2)(1)}$;
thus,  $f\1(S(f\2)(x)y) = x (f\1(S(f\2)(1)y))$ and
\begin{align*}
f(xy) = f\1(x) f\2(S(f\3)(1)y) = f\1(x) (f\2(1) \backslash f\3(y)).
\end{align*}
Since  $\rho_u(x) = R_u(x)$, $\rho(u,v)(x) = r(u,v)(x)$  and $f(x) = H(u,v)(x)$, we are done.
\end{proof}
%
\section{Pseudoreductive Hopf algebras}
\label{sec:pHa}
In this section we introduce pseudoreductive Hopf algebras by analogy with pseudoreductive loops. We will describe pseudoreductive triples associated to these Hopf algebras, and finally we will formally integrate pseudoreductive triples to obtain pseudoreductive Hopf algebras. Our approach sticks again to the work of Sabinin on analytic loops, and we cannot avoid his clever use of the formula for the differential of the exponential map without  resorting to artificial arguments.

\subsection{Pseudoreductive loops, triples and Hopf algebras.}
A loop $Q$ is called (right) \emph{pseudoreductive} if the following properties hold:
\begin{itemize}
	\item \emph{Right pseudospeciality}: there exists a map $\bullet\colon Q \times Q \rightarrow Q$ such that $e \bullet u = e = u \bullet e$ ($e$ stands for the identity element of $Q$) and 
	\begin{displaymath}
		P(u,v)(xy) = r(u,v)(x) P(u,v)(y)
	\end{displaymath}
	where $P(u,v)(x) := r(u,v)(x)(u \bullet v)$.
	\item \emph{Right monoalternativity}: $((xy)y \cdots )y  = x((yy) \cdots y)$ for all $n \geq 0$, where $y$ appears $n$ times in each side of the equality.
\end{itemize}
\begin{definition}
A  non-associative Hopf algebra $U$ is called (right) \emph{pseudoreductive} if the following properties hold: 
\begin{itemize}
	\item \emph{Right pseudospeciality}: there exists a coalgebra morphism $\bullet \colon U \otimes U \rightarrow U$ such that $1 \bullet u = \epsilon(u)1 = u \bullet 1 $ and
	\begin{displaymath}
			P(u,v) (xy) = 
			 r(u\1,v\1)(x) P(u\2,v\2)(y)
	\end{displaymath}
	where $P(u,v)(x) := r(u\1,v\2)(x) (u\2 \bullet v\2)$.
	\item \emph{Right monoalternativity}:  $(((xy\1)y\2)\cdots )y_{(n)} = x ((y\1 y\2)\cdots y_{(n)})$ for all $n \geq 0$.
\end{itemize}	
\end{definition}
\begin{definition}
We will say that a triple $\tau = (\fg, \fs, \rho_{\fc})$ is \emph{pseudoreductive} if there exists a map $\zeta \colon \fs \to \rho_{\fc}$ such that

\medskip

	(PRT1) $d + \zeta(d) \in N_{\fg}(\rho_{\fc})$ for all $d \in \fs$, and
	
\medskip

	(PRT2) $\ad^{2n}_{\rho_c} (\zeta(\fs)) \subseteq \rho_{\fc}$ for all $n \geq 0$ and $\rho_c\in \rho_{\fc}$.
\end{definition}
Notice that the first condition is equivalent to $\tau$ being hyporeductive.
\begin{lemma}
Let $U$ be a pseudoreducitve Hopf algebra. For any $a, b \in \Prim(U)$, $a \bullet b$ is also primitive.
\end{lemma}
\begin{proof}
See the proof of Lemma~\ref{lem:primitive}.
\end{proof}
\begin{proposition}\label{prop:pseudo_is_hypo}
Let $U$ be a pseudoreductive Hopf algebra. For any $a, b, c \in \Prim(U)$ we have
\begin{equation*}
	[r(a,b) + \frac{1}{2} R_{a \bullet b}, R_c]  = R_{r(a,b)(c) + \frac{1}{2}[c, a \bullet b]}.
\end{equation*}
Hence, $\T(U)\sinn$ is a hyporeductive triple.
\end{proposition}
\begin{proof}
Fix $a, b, c \in \Prim(U)$. On the one hand, $P(a,b) = r(a,b) + R_{a \bullet b}$ and $P(a,b) R_y = R_y r(a,b) + R_{P(a,b)(y)}$ so
\begin{align*}
	[r(a,b) + \frac{1}{2} R_{a \bullet b}, R_c] &= - R_{a \bullet b}R_c + R_{P(a,b)(c)} + \frac{1}{2} [R_{a \bullet b}, R_c] \\
	&= - \frac{1}{2} R_{a \bullet b} R_c - \frac{1}{2} R_c R_{a \bullet b} + R_{P(a,b)(c)};
\end{align*}
on the other hand, the monoalternativity implies $R_{a \bullet b} R_c + R_c R_{a \bullet b} = R_{(a \bullet b)c + c (a \bullet b)}$, i.e. $[r(a,b) + \frac{1}{2} R_{a \bullet b}, R_c]$ is the right multiplication operator by the primitive element ${r(a,b)(c) + \frac{1}{2}[c, a \bullet b]} = [r(a,b) + \frac{1}{2} R_{a \bullet b}, R_c] (1)$.
\end{proof}
The following result highlights the property behind the axiom (PRT2) in the definition of pseudoreductive triples. 

%
\begin{proposition}
Let $U$ be a pseudoreductive connected Hopf algebra. Then
\begin{equation} \label{eq:Bol}
	R_{a \bullet b} R_y + R_y R_{a \bullet b} = R_{(a \bullet b) y + y (a \bullet b)}
\end{equation}
for all $a, b \in \Prim(U)$ and $y \in U$.
\end{proposition}
\begin{proof}
By Proposition~ \ref{prop:pseudo_is_hypo} we have $[r(a,b) + \frac{1}{2} R_{a  \bullet b}, R_c] = R_{r(a,b)(c) + \frac{1}{2}[c, a \bullet b]}$, and  by monoalternativity $[r(a,b) + \frac{1}{2} R_{a \bullet b}, R_{c^n}] \in R_{U}$ for all $a,b,c \in \Prim(U)$. The set $\{ c^n \mid n \geq 0, \, c \in \Prim(U) \}$ spans $U$ ($U$ is connected), thus
\begin{equation*}
	-\frac{1}{2} (R_{a \bullet b} R_y + R_y R_{a \bullet b}) + R_{P(a,b)(y)} = [r(a,b) + \frac{1}{2} R_{a \bullet b}, R_{y}] \in R_{U}
\end{equation*}
for all $y \in U$. This proves $R_{a \bullet b} R_y + R_y R_{a \bullet b} \in R_U$. After evaluating this operator at $1$ we get the result.
\end{proof}
We can use  (\ref{eq:Bol}) to prove that $T(U)\sinn$ satisfies (PRT2) for the initial value $n = 1$. For any $a, b, c \in \Prim(U)$,
\begin{align*}
[R_c,[R_c,R_{a \bullet b}]] &= R^2_c R_{a \bullet b} - 2 R_c R_{a \bullet b} R_c + R_{a \bullet b}R^2_c\\
&= 3(R_{c^2}R_{a \bullet b} + R_{a \bullet b}  R_{c^2}) - 2( R_{a \bullet b} R_c R_c + R_c R_{a \bullet b} R_c + R_c R_c R_{a \bullet b}).
\end{align*}
Both expressions in parentheses are multiplication operators, so $[R_c,[R_c,R_{a \bullet b}]] \in \fc(U)$.  Unfortunately it is much more difficult to prove that $\T(U)\sinn$ is a pseudoreductive triple if we use this approach.  Sabinin \cite{Sa99} gave a short proof for analytic loops. It is convenient to fix two parameters $s,t$ so that we have at our disposal the algebra of formal power series $A[[s,t]]$ with coefficients in another algebra $A$. If $A$ is associative and unital, expressions such as $\exp(ta)$ make sense in $A[[s,t]]$. If $A$ is non-associative then although $\exp_l(ta):= 1 + ta + \frac{1}{2!}(ta)(ta) + \frac{1}{3!}((ta)(ta))(ta) + \cdots = \sum_{n=0}^\infty \frac{1}{n!} (ta)^n$ requires our convention about powers, for the most part, it is again a well-defined element in $A[[s,t]]$.
\begin{proposition}\label{prop:pseudo}
	Let $U$ be a pseudoreductive Hopf algebra. Then, $T(U)\sinn$ is a pseudoreductive triple.
\end{proposition}
\begin{proof}
Identity (\ref{eq:Bol}) implies
\begin{equation*}
\sum_{i = 0}^n \binom{n}{i} R^i_{a \bullet b} R_y R^{n-i}_{a \bullet b} \in R_U
\end{equation*}
for any $a, b \in \Prim(U)$ and $y \in U$. Thus, $\exp(t R_{a \bullet b}) \exp(sR_c) \exp(t R_{a \bullet b}) \in  R_{U[[s,t]]}$
for all $a, b, c \in \Prim(U)$. Evaluating at $1$ we obtain
\begin{equation}\label{eq:Bruck}
	\exp(t R_{a \bullet b}) \exp(sR_c) \exp(t R_{a \bullet b})  = R_{(\exp_l(t a \bullet b) \exp_l(sc))\exp_l(t a \bullet b)}.
\end{equation}
The element  ($\exp_l(t a \bullet b) \exp_l(sc)) \exp_l(t a \bullet b)$ is group-like (see \cite{MoPeSh16a} for a discussion on group-like elements), so there exists $c(s,t) = c(s,t;a,b) \in \Prim(U)[[s,t]]$ such that $f(s,t) := R_{c(s,t)}$ satisfies
\begin{equation*}
 \exp(t R_{a \bullet b}) \exp(sR_c) \exp(t R_{a \bullet b})  = \exp(f(s,t)).
\end{equation*}
Taking derivatives with respect to $t$ we have\footnote{Consider the free associative algebra $\field\langle x,x' \rangle$ and the derivation $d:= x'\frac{\partial}{\partial_x}$. With the natural Hopf algebra structure ($x,x'$ are primitive) the map $\gamma(u) = S(u\1) d(u\2)$ satisfies $\gamma(ux) = [\gamma(u),x] + \epsilon(u)x'$. By allowing formal power series, we easily obtain $\exp(-x) d(\exp(x)) = \gamma(\exp(x)) = \frac{\exp(-\ad_{x})-\Id}{-\ad_x}(x')$.}
\begin{align*}
\frac{\exp(- \ad_{f(s,t)}) - \Id}{- \ad_{f(s,t)}} \left(\frac{\partial f}{\partial t}\right) &= \exp( - f(s,t)) \frac{\partial \exp(f(s,t))}{\partial t}\\
&= \exp(-f(s,t)) R_{a \bullet b} \exp(f(s,t)) + R_{a \bullet b} \\
&= \exp(- \ad_{f(s,t)})(R_{a \bullet b}) + R_{a \bullet b}.
\end{align*}
Therefore,
\begin{align*}
	\frac{\partial f}{\partial t} &= \frac{\ad_{f(s,t)}}{\exp(\ad_{f(s,t)}) - \Id} (R_{a \bullet b}) + \frac{-\ad_{f(s,t)}}{\exp(- \ad_{f(s,t)}) - \Id}(R_{a \bullet b})\\
	&= 2 \sum_{2m \geq 0} \frac{\beta_{2m}}{(2m)!} \ad^{2m}_{f(s,t)}(R_{a \bullet b}) 
\end{align*}
where $0\neq \beta_{2m}$ is the  $2m$-th Betti number. Evaluating at $t = 0$, since $f(s,0 ) = R_{sc}$ and $f \in R_{U[[s,t]]}$, we obtain $\ad^{2m}_{R_{c}}(R_{a \bullet b}) \in \fc(U)$ for all $m \geq 0$ and $a, b , c \in \Prim(U)$. This proves the pseudospeciality of $\T(U)\sinn$. 
\end{proof}
Equality (\ref{eq:Bruck}) naturally appears in the context of Bruck loops and symmetric homogeneous spaces. These structures are related to Lie triple systems, i.e. subspaces of Lie algebras closed under the double commutator. Axiom (PRT2) evokes these structures, and our presentation aimed for it.

\subsection{Formal integration of pseudoreductive triples}
\begin{theorem}
Let $\tau$ be a pseudoreductive triple. Then, $U_{\tau}$ is a pseudoreductive Hopf algebra.
\end{theorem}
\begin{proof}
Let $\tau = ( \fg, \fs, \rho_{\fc})$ be a pseudoreductive triple. The pseudoreductive structure of $\tau$ induces via  $\iota \colon \tau  \to \T(U_{\tau})$ a corresponding pseudoreductive structure on the triple $\T(U_{\tau})\sinn$. Thus, there exists a map $\zeta \colon \fs(U_{\tau})\sinn \to \fc(U_{\tau})$ such that $d + \zeta(d) \in N_{\fg(U_{\tau})\sinn}(\fc(U_{\tau}))$. Given $a, b \in \Prim(U_{\tau})$ we define $a \bullet b$ by
\begin{equation*}
	R_{a \bullet b} := \zeta( 2r(a,b)).
\end{equation*}
With this choice we get $[r(a,b) + \frac{1}{2} R_{a \bullet b}, R_c] = R_{r(a,b)(c) + \frac{1}{2}[c, a \bullet b]}$ for any $c \in \Prim(U_{\tau})$. The monoalternativity makes this identity valid for all $y \in U_{\tau}$, i.e.
\begin{equation*}
	[r(a,b) + \frac{1}{2}R_{a \bullet b}, R_y] = R_{r(a,b)(y) + \frac{1}{2}[y, a \bullet b]}.
\end{equation*}
Let us consider the function $f(s,t)$ in  
\begin{equation*}
 	\exp(t R_{a \bullet b}) \exp(sR_c) \exp(t R_{a \bullet b})  = \exp(f(s,t))
\end{equation*}
used in the proof of Proposition~ \ref{prop:pseudo}. Beware, now we cannot assume $f(s,t) \in \fc(U_{\tau})[[s,t]]$.
However, this function is uniquely determined by 
\begin{equation*}
	\frac{\partial f}{\partial t} = 2 \sum_{2m \geq 0} \frac{\beta_{2m}}{(2m)!} \ad^{2m}_{f(s,t)}(R_{a \bullet b})
\end{equation*}
and the inital condition $f(s,0) = R_{sc} \in \fc(U_{\tau})[[s,t]]$, which by  (PRT2)  ensures $f(s,t) \in \fc(U_{\tau})[[s,t]]$ and $\exp(f(s,t)) \in R_{U_{\tau}[[s,t]]}$. In particular, 
\begin{equation*}
	\frac{\partial\exp(f)}{\partial t}(s,0) \in R_{U_{\tau}[[s]]} 
\end{equation*}
implies
\begin{equation*}
	R_{a \bullet b} R_y + R_y R_{a \bullet b} = R_{(a \bullet b) y + y (a \bullet b)}.
\end{equation*}
for all $a, b \in \Prim(U_{\tau})$ and $y \in U_{\tau}$. Therefore, if we define $P(a,b) := r(a,b) + R_{a \bullet b}$ then 
\begin{equation}\label{eq:ternary_derivation}
P(a,b)(xy) = r(a,b)(x)y + x P(a,b)(y).
\end{equation}
Let $\fp$ be the Lie algebra generated by $\{ P(a,b) \mid a, b \in \Prim(U_{\tau}) \}$ and let  $\theta$ be a linear map $U(\fs(U_{\tau})\sinn) \to \fp$  such that $\theta(d) =  d + 2 \zeta(d)$ for all $d \in \fs(U_{\tau})\sinn$, i.e. $\theta(r(a,b)) =  P(a,b)$, and $\theta(1) = 0$. We extend $\theta$ to a coalgebra morphism $\theta' \colon U(\fs(U_{\tau})\sinn) \to U(\fp)$ by $\theta'= \exp^*(\theta)$ as in the proof of Lemma~\ref{lem:section}. We also define the isomorphism $\Psi := \pi \theta'$ and  elements
\begin{equation*}
	u \bullet v := \theta'\Psi^{-1}(r(u,v))(1).
\end{equation*}
Any operator $\theta' \Psi(r(u,v)) = R_{u\1 \bullet v\1} r(u\2,v\2)$ belongs to  $U(\fp)$, so (\ref{eq:ternary_derivation}) implies
\begin{equation*}
R_{u\1 \bullet v\1} r(u\2, v\2)(xy) = r(u\1, v\1)(x) R_{u\2 \bullet v\2)} r(u\3, v\3)(y)
\end{equation*}
as desired.
\end{proof}

%
\section{Hyporeductive triple algebras}
\label{sec:hta}
Hyporeductive loops can also be constructed from hyporeductive triple algebras. In this section we discuss these structures from the point of view of Sabinin algebras.

\subsection{Hyporeductive triple algebras} The origin of \emph{hyporeductive triple algebras} is linked to the following construction. Let $\tau = (\fg, \fs, \fc)$ be a hyporeductive triple (here it would be more convenient to return to the original notation $\fc$ instead of $\rho_{\fc}$), so $\fg = \fn + \fc$ where $\fn = N_{\fg}(\fc)$. We can choose a subspace $\fh \subseteq \fn$ such that $\fg = \fh \oplus \fc$. This decomposition induces two binary products $a \cdot b$, $a*b$ and a triple product $[c;a,b]$ on $\fc$ as follows
\begin{align}\label{eq:hrta}
		[a,b] &= h(a,b) + a \cdot b \nonumber\\
		h(a,b) &= s(a,b) + a*b \\
		[h(a,b),c] &= [c;a,b] \nonumber
\end{align}
with $h(a,b) \in \fh$, $s(a,b) \in \fs$ and $a,b,c \in \fc$. The new algebraic structure $(\fc, a \cdot b, a* b, [c; a,b])$  is characterized by twelve axioms (apart from the obvious skew-commutativity) coming from the Jacobi identity \cite{Is10}. Conversely, given a hyporeductive triple algebra $(\fc, a \cdot b, a* b, [c; a,b])$ there exists a Lie algebra $E(\fc)$ generated by $\fc$ with relations
\begin{equation*}
[[a,b] - a \cdot b, c] = [c; a,b]
\end{equation*}
for any $a, b, c \in \fc$. This Lie algebra splits as $E(\fc) = \fs  \oplus \fc$ with $\fs := \spann \langle [a,b] - a \cdot b \mid a, b \in \fc \rangle$ being a subalgebra. Thus, $(E(\fc), \fs, \fc)$ is a hyporeductive triple \cite{Is10}. The hyporeductive triple algebra  $(\fc, a \cdot b, a* b, [c; a,b])$  is recovered from this triple by (\ref{eq:hrta}).  Unfortunately, the construction of  hyporeductive triple algebras from hyporeductive triples  involves the choice of $\fh$.

The formal integration of a hyporeductive triple algebra $(\fc, a \cdot b, a*b, [c;a,b])$ might be understood as the construction of a hyporeductive (connected) Hopf algebra $U(\fc)$ such that $\fc = \Prim(U(\fc))$ and
\begin{equation} \label{eq:universal_property}
[a,b] = a*b + a \cdot b \quad \text{and}\quad [[R_a,R_b] + R_{a \cdot b}, R_c] = R_{[c;a,b]}.
\end{equation}
It is not difficult to check that if we begin with the Lie algebra $E(\fc)\supopp$ (recall that $\fg\supopp$ has the opposite product $-[x,y]$ instead of the product $[x,y]$ of $\fg$) and the triple $(E(\fc)\supopp, \fs\supopp, \fc)$ then the corresponding hyporeductive Hopf algebra associated to this triple satisfies the required properties.

\subsection{Sabinin algebras associated to hyporeductive triple algebras} \label{subsec:Sabinin} As discussed at the beginning of this paper, the algebraic structure of the tangent space of any analytic loop is known as Sabinin algebra. For general loops, two families of multilinear operations $\langle x_1,\dots, x_m; y,z \rangle$, $\Phi(x_1,\dots, x_m; y_1,\dots, y_n;y_{n+1})$ take part in the description of the corresponding Sabinin tangent algebra. However, for monoalternative loops only the operations $\langle x_1,\dots, x_n; y,z \rangle$, $n \geq 0$ are required since $\Phi(x_1,\dots, x_m; y_1,\dots, y_n;y_{n+1}) = 0$ for these loops. The connection of hyporeductive triple algebras with the infinitesimal study of hyporeductive loops suggests that a Sabinin algebra can be obtained out of any such triple algebra. The ``underline'' notation from Section~\ref{subsec:SU_operations} is quite useful for that.

\begin{theorem}
Let $(\fc, a \cdot b, a *b, [c;a,b])$ be a hyporeductive triple algebra. The multilinear operations
\begin{align*}
	\langle 1; y,z \rangle &:= - y \cdot z - z*y \\
	\langle\underline{u{x_{m+1}}}; y,z \rangle &:= \langle \underline{u}; x_{m+1}, y \cdot z \rangle + \langle \underline{u}\1 ; x_{m+1}, \langle \underline{u}\2; y,z \rangle \rangle + \epsilon(u) [x_{m+1}; y,z]
\end{align*}
define a Sabinin algebra structure on $\fc$.
\end{theorem}
\begin{proof}
Let $U(\fc)$ be the hyporeductive Hopf algebra that formally integrates $(\fc, a \cdot b, a *b, [c;a,b])$. We will compute the Shestakov-Umirbaev operations   (recall Section~\ref{subsec:SU_operations}) $\langle -; -,-\rangle$ of $\SU(U(\fc))$ to obtain the desired Sabinin algebra $(\fc, \langle - ; - , - \rangle)$.  By (\ref{eq:universal_property}) we have
\begin{equation*}
	\begin{split}
		&\{ ((u x_{m+1})z)y - ((u x_{m+1})y)z - (u x_{m+1})(y \cdot z)\}  \\ 
		&\quad - \{ ((u z)y) x_{m+1} - ((u y) z) x_{m+1} + (u(y \cdot z)) x_{m+1} \} = u[x_{m+1};y,z],
	\end{split}
\end{equation*}
for any $x_1,\dots, x_{m}, x_{m+1}, y, z \in \fc$, $u =((x_1x_2)\cdots)x_m$ and $\underline{u} = x_1 \otimes \cdots \otimes x_m$,  so
\begin{equation*}
	\begin{split}
		& (u x_{m+1})\1 \langle \underline{(u x_{m+1})}\2; y,z \rangle \\ 
		& \quad + u\1 \langle \underline{u}\2; y \cdot z, x_{m+1} \rangle - (u\1  \langle u\2; y, z\rangle) x_{m+1} = u[x_{m+1};y,z],
	\end{split}
\end{equation*}
which implies
\begin{equation*}
	\begin{split}
		& u\1 \langle \underline{u}\2; \langle \underline{u}\3; y,z \rangle, x_{m+1} \rangle \\
		& \quad  + u\1 \langle \underline{u\2 x_{m+1}}; y, z \rangle + u\1 \langle \underline{u}\2; y \cdot z, x_{m+1} \rangle  = u[x_{m+1};y,z].
	\end{split}
\end{equation*}
Dividing by $u\1$ we get
\begin{equation*}
\langle \underline{u x_{m+1}}; y,z \rangle = \langle \underline{u}; x_{m+1}, y \cdot z \rangle + \langle \underline{u}\1; x_{m+1}, \langle \underline{u}\2; y,z \rangle \rangle + \epsilon(u) [x_{m+1};y,z],
\end{equation*}
and, by definition of $\langle 1; y,z \rangle$ and the properties of $U(\fc)$, $\langle 1; y,z \rangle = - y * z - y \cdot z$. 
\end{proof}


\end{document}